\documentclass[12pt]{amsart}

\usepackage[utf8x]{inputenc}
\usepackage[english]{babel}
\usepackage{amsmath, amsfonts, amssymb}
\usepackage{color}
\usepackage{tikz}
\usetikzlibrary{positioning,shapes,shadows,arrows}
\usepackage{enumitem}
\usepackage{hhline}
\usepackage{array}

\title{Counting smaller trees in the Tamari order}

\author{Gr\'egory Chatel, Viviane Pons}

\address{Laboratoire d'Informatique Gaspard Monge, Université Paris-Est
    Marne-la-Vallée, 5 Boulevard Descartes, Champs-sur-Marne,
    77454 Marne-la-Vallée cedex 2, France}
    

\newtheorem{Theoreme}{Theorem}[section]

\newtheorem{Proposition}[Theoreme]{Proposition}

\newtheorem{Definition}[Theoreme]{Definition}

\DeclareMathOperator{\PBT}{\textbf{PBT}}
\DeclareMathOperator{\HT}{\textbf{H}}
\DeclareMathOperator{\ET}{\textbf{E}}
\DeclareMathOperator{\PT}{\textbf{P}}

\newcommand\dec{F_{\ge}}
\newcommand\inc{F_{\le}}

\newcolumntype{C}{>{\centering\arraybackslash} m{3.2cm}}
\newcolumntype{D}{>{\centering\arraybackslash} m{3.5cm}}

\begin{document}

\maketitle

\begin{abstract}
We introduce new objects, the interval-posets, that encode intervals of the Tamari lattice. We then find a combinatorial interpretation of the bilinear form that appears in the functional equation of Tamari intervals described by Chapoton. Thus, we retrieve this functional equation and prove that the polynomial recursively computed from the bilinear form on each tree $T$ counts the number of trees smaller than $T$ in the Tamari order.

Nous introduisons un nouvel objet, les intervalles-posets, pour encoder les intervalles de Tamari. Nous donnons ainsi une interprétation combinatoire à la forme bilinéaire qui apparaît dans l'équation fonctionnelle des intervalles de Tamari que donne Chapoton. De cette façon, nous retrouvons d'une nouvelle manière cette équation fonctionnelle et prouvons que le polynôme calculé récur-sivement à partir de la forme bilinéaire pour chaque arbre $T$ compte le nombre d'arbres plus petits que $T$ dans l'ordre de Tamari.
\end{abstract}

\section{Introduction}
\label{sec:Intro}

The combinatorics of planar binary trees has already being linked with interesting algebraic properties. Loday and Ronco first introduced the $\PBT$ Hopf Algebra based on these objects \cite{PBT1}. It was re-constructed by Hivert, Novelli and Thibon \cite{plactic-monoid} through the introduction of the sylvester monoid. The structure of $\PBT$ involves a very nice object which is linked to both algebra and classical algorithmic: the \emph{Tamari lattice}. 

This order on binary trees is based on the \emph{right rotation} operation (see Figure \ref{fig:tree-right-rotation}), commonly used in sorting algorithms through binary search trees. The lattice itself first appeared in the context of the associahedron \cite{Tamari}. As its vertices are counted by Catalan numbers, the covering relations can be described by many combinatorial objects \cite{Stanley} two common ones being planar binary trees and Dyck paths. Recently Chapoton gave a formula for the number of intervals \cite{Chap}:

\begin{equation}
\label{eq:intervals-formula}
I_n = \frac{2(4n+1)!}{(n+1)!(3n+2)!},
\end{equation}
where $I_n$ is the number of intervals of the Tamari lattice of binary trees of size $n$. This formula was very recently generalized to a new set of lattices, the $m$-Tamari \cite{mTamari}.

It has been known since Bj\"orner and Wachs \cite{BW} that linear extensions of a certain labelling of binary trees correspond to intervals of the weak order on permutations. This was more explicitly described in \cite{plactic-monoid} with sylvester classes. The elements of the basis $\PT$ of $\PBT$ are defined as a sum on a sylvester class of elements of $\textbf{FQSym}$. The $\PBT$ algebra also admits two other bases $\HT$ and $\ET$ which actually correspond to respectively initial and final intervals of the Tamari order. They can be indexed by plane forests and, with a well chosen labelling, their linear extensions are intervals of the weak order on permutations corresponding to a union of sylvester classes. In this paper, we introduce a more general object, the Tamari interval-poset, which encodes a general interval of the Tamari lattice and whose linear extensions are exactly the corresponding sylvester classes (and so an interval of the weak order). This new object has nice combinatorial properties and allows for computation on Tamari intervals.

Thereby, we give a new proof of the formula of Chapoton \eqref{eq:intervals-formula}. This proof is based on the study of a bilinear form that already appeared in \cite{Chap} but was not explored yet. It leads to the definition of a new family of polynomials:

\begin{Definition}
\label{def:tamari-polynomials}
Let $T$ be a binary tree, the polynomial $\mathcal{B}_T(x)$ is recursively defined by
\begin{align}
\mathcal{B}_\emptyset &:= 1 \\
\mathcal{B}_T(x) &:= x \mathcal{B}_L(x) \frac{x \mathcal{B}_R(x) - \mathcal{B}_R(1)}{x - 1}
\end{align}
where $L$ and $R$ are respectively the left and right subtrees of $T$. We call $\mathcal{B}_T(x)$ the \emph{Tamari polynomial} of $T$ and the \emph{Tamari polynomials} are the set of all polynomials obtained by this process.
\end{Definition}

This family of polynomials is yet unexplored in this context but a different computation made by Chapoton \cite{ChapBiVar} on rooted trees seems to give a bivariate version. Our approach on Tamari interval-posets allows us to prove the following theorem:

\begin{Theoreme}
\label{thm:smaller-trees}
Let $T$ be a binary tree. Its Tamari polynomial $\mathcal{B}_T(x)$ counts the trees smaller than $T$ in the Tamari order according to the number of nodes on their left border. In particular, $\mathcal{B}_T(1)$ is the number of trees smaller than $T$.

Symmetrically, if $\tilde{\mathcal{B}}_T$ is defined by exchanging the role of left and right children in Definition \ref{def:tamari-polynomials}, then it counts the number of trees greater than $T$ according to the number of nodes on their right border.
\end{Theoreme}

This theorem will be proven in Section \ref{sec:main-result}. In Section \ref{sec:def}, we recall some definitions and properties of the Tamari lattice and introduce the notion of \emph{interval-poset} to encode a Tamari interval. In Section \ref{sec:tamari-polynomials}, we show the implicit bilinear form that appears in the functional equation  of the generating functions of Tamari intervals. We then explain how interval-posets can be used to give a combinatorial interpretation of this bilinear form and thereby give a new proof of the functional equation. Theorem \ref{thm:smaller-trees} follows naturally. In Section \ref{sec:final}, we give two independent contexts in which our problem can be generalized: flows of rooted trees and $m$-Tamari intervals.

\section{Definitions of Tamari interval-posets}
\label{sec:def}

\subsection{Binary trees and Tamari order}
\label{sec:tamari}

A binary tree is recursively defined by being either the empty tree $(\emptyset)$ or a pair of  binary trees, respectively called \emph{left} and \emph{right} subtrees, grafted on an internal node. If a tree $T$ is composed of a root node $x$ with $A$ and $B$ as respectively left and right subtrees, we write $T = x(A,B)$. The number of nodes of a tree $T$ is called the size of $T$. The Tamari order is an order on trees of a given size using the \emph{rotation} operation.

\begin{Definition}
Let $y$ be a node of $T$ with a non-empty left subtree $x$. The \emph{right rotation} of $T$ on $y$ is a local rewriting which follows Figure~\ref{fig:tree-right-rotation}, that is replacing $y( x(A,B), C)$ by $x(A,y(B,C))$ (note that $A$, $B$, or $C$ might be empty).

\begin{figure}[ht]
\centering
\scalebox{0.7}{
\begin{tikzpicture}
\node(TX1) at (-3,0){x};
\node(TY1) at (-2,1){y};
\node(TA1) at (-3.5,-1){A};
\node(TB1) at (-2.5,-1){B};
\node(TC1) at (-1,0){C};
\node(to) at (0,0){$\to$};
\node(TX2) at (2,1){x};
\node(TY2) at (3,0){y};
\node(TA2) at (1,0){A};
\node(TB2) at (2.5,-1){B};
\node(TC2) at (3.5,-1){C};

\draw (TA1) -- (TX1);
\draw (TB1) -- (TX1);
\draw (TC1) -- (TY1);
\draw (TX1) -- (TY1);

\draw (TA2) -- (TX2);
\draw (TB2) -- (TY2);
\draw (TC2) -- (TY2);
\draw (TY2) -- (TX2);
\end{tikzpicture}
}
\caption{Right rotation on a binary tree.}

\label{fig:tree-right-rotation}

\end{figure}
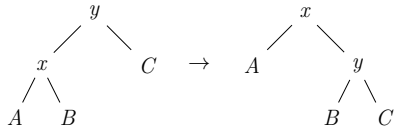

\end{Definition}

The Tamari order is the transitive and reflexive closure of the right
rotation: a tree $T'$ is greater than a tree $T$ if $T'$ can be
obtained by applying a sequence of right rotations on $T$. It is
actually a lattice \cite{Tamari}, see Figure \ref{fig:tamari-order}
for some examples. One of the purposes of this article is to define a
combinatorial object that would correspond to Tamari intervals.

\begin{figure}[ht]

\centering
\scalebox{0.8}{
\begin{tabular}{cc}
\begin{tikzpicture}
\node(T1) at (1,0) {
  \scalebox{0.7}{
    \begin{tikzpicture}
      \node(N1) at (0,-1){$\bullet$};
      \node(N2) at (0.5,-0.5){$\bullet$};
      \node(N3) at (1,0){$\bullet$};
      
      \draw (N1.center) -- (N2.center);
      \draw (N2.center) -- (N3.center);
    \end{tikzpicture}
  }
};
\node(T2) at (0,-2.2) {
  \scalebox{0.7}{
    \begin{tikzpicture}
      \node(N1) at (0,-0.5){$\bullet$};
      \node(N2) at (0.5,0){$\bullet$};
      \node(N3) at (1,-0.5){$\bullet$};
      
      \draw (N1.center) -- (N2.center);
      \draw (N3.center) -- (N2.center);
    \end{tikzpicture}
  }
};

\node(T3) at (2,-1.5) {
  \scalebox{0.7}{
    \begin{tikzpicture}
      \node(N1) at (0,-0.5){$\bullet$};
      \node(N2) at (0.5,-1){$\bullet$};
      \node(N3) at (0.5,0){$\bullet$};
      
      \draw (N1.center) -- (N3.center);
      \draw (N2.center) -- (N1.center);
    \end{tikzpicture}
  }
};

\node(T4) at (2,-3) {
  \scalebox{0.7}{
    \begin{tikzpicture}
      \node(N1) at (0,0){$\bullet$};
      \node(N2) at (0,-1){$\bullet$};
      \node(N3) at (0.5,-0.5){$\bullet$};
      
      \draw (N3.center) -- (N1.center);
      \draw (N2.center) -- (N3.center);
    \end{tikzpicture}
  }
};

\node(T5) at (1,-4.5) {
  \scalebox{0.7}{
    \begin{tikzpicture}
      \node(N1) at (0,0){$\bullet$};
      \node(N2) at (0.5,-0.5){$\bullet$};
      \node(N3) at (1,-1){$\bullet$};
      
      \draw (N3.center) -- (N2.center);
      \draw (N2.center) -- (N1.center);
    \end{tikzpicture}
  }
};

\draw (T1) -- (T2);
\draw (T1) -- (T3);
\draw (T3) -- (T4);
\draw (T4) -- (T5);
\draw (T2) -- (T5);

\end{tikzpicture}
&
\begin{tikzpicture}
\node(T1) at (4,0){
	\scalebox{0.5}{
	\begin{tikzpicture}
		\node(N1) at (0,-1.5){$\bullet$};
		\node(N2) at (0.5,-1){$\bullet$};
		\node(N3) at (1,-0.5){$\bullet$};
		\node(N4) at (1.5,0){$\bullet$};
		
		\draw (N3.center) -- (N4.center);
		\draw (N2.center) -- (N3.center);
		\draw (N1.center) -- (N2.center);
	\end{tikzpicture}
	}
};
\node(T2) at (2,-0.5){
	\scalebox{0.5}{
	\begin{tikzpicture}
		\node(N1) at (0,-1){$\bullet$};
		\node(N2) at (0.5,-1.5){$\bullet$};
		\node(N3) at (1,-0.5){$\bullet$};
		\node(N4) at (1.5,0){$\bullet$};
		
		\draw (N3.center) -- (N4.center);
		\draw (N1.center) -- (N3.center);
		\draw (N2.center) -- (N1.center);
	\end{tikzpicture}
	}
};

\node(T3) at (6,-1){
	\scalebox{0.5}{
	\begin{tikzpicture}
		\node(N1) at (0,-1){$\bullet$};
		\node(N2) at (0.5,-0.5){$\bullet$};
		\node(N3) at (1,-1){$\bullet$};
		\node(N4) at (1.5,0){$\bullet$};
		
		\draw (N2.center) -- (N4.center);
		\draw (N1.center) -- (N2.center);
		\draw (N3.center) -- (N2.center);
	\end{tikzpicture}
	}
};

\node(T4) at (0,-1){
	\scalebox{0.5}{
	\begin{tikzpicture}
		\node(N1) at (0,-0.5){$\bullet$};
		\node(N2) at (0.5,-1.5){$\bullet$};
		\node(N3) at (1,-1){$\bullet$};
		\node(N4) at (1.5,0){$\bullet$};
		
		\draw (N1.center) -- (N4.center);
		\draw (N3.center) -- (N1.center);
		\draw (N2.center) -- (N3.center);
	\end{tikzpicture}
	}
};

\node(T5) at (8,-2){
	\scalebox{0.5}{
	\begin{tikzpicture}
		\node(N1) at (0,-0.5){$\bullet$};
		\node(N2) at (0.5,-1){$\bullet$};
		\node(N3) at (1,-1.5){$\bullet$};
		\node(N4) at (1.5,0){$\bullet$};
		
		\draw (N1.center) -- (N4.center);
		\draw (N2.center) -- (N1.center);
		\draw (N3.center) -- (N2.center);
	\end{tikzpicture}
	}
};

\node(T6) at (4,-2.5){
	\scalebox{0.5}{
	\begin{tikzpicture}
		\node(N1) at (0,-1){$\bullet$};
		\node(N2) at (0.5,-0.5){$\bullet$};
		\node(N3) at (1,0){$\bullet$};
		\node(N4) at (1.5,-0.5){$\bullet$};
		
		\draw (N4.center) -- (N3.center);
		\draw (N2.center) -- (N3.center);
		\draw (N1.center) -- (N2.center);
	\end{tikzpicture}
	}
};

\node(T7) at (2,-3.5){
	\scalebox{0.5}{
	\begin{tikzpicture}
		\node(N1) at (0,-0.5){$\bullet$};
		\node(N2) at (0.5,-1){$\bullet$};
		\node(N3) at (1,0){$\bullet$};
		\node(N4) at (1.5,-0.5){$\bullet$};
		
		\draw (N4.center) -- (N3.center);
		\draw (N1.center) -- (N3.center);
		\draw (N1.center) -- (N2.center);
	\end{tikzpicture}
	}
};

\node(T8) at (6,-3){
	\scalebox{0.5}{
	\begin{tikzpicture}
		\node(N1) at (0,-0.5){$\bullet$};
		\node(N2) at (0.5,0){$\bullet$};
		\node(N3) at (1,-1){$\bullet$};
		\node(N4) at (1.5,-0.5){$\bullet$};
		
		\draw (N4.center) -- (N2.center);
		\draw (N1.center) -- (N2.center);
		\draw (N3.center) -- (N4.center);
	\end{tikzpicture}
	}
};

\node(T9) at (4,-3.9){
	\scalebox{0.5}{
	\begin{tikzpicture}
		\node(N1) at (0,-0.5){$\bullet$};
		\node(N2) at (0.5,0){$\bullet$};
		\node(N3) at (1,-0.5){$\bullet$};
		\node(N4) at (1.5,-1){$\bullet$};
		
		\draw (N1.center) -- (N2.center);
		\draw (N3.center) -- (N2.center);
		\draw (N4.center) -- (N3.center);
	\end{tikzpicture}
	}
};

\node(T10) at (0,-4){
	\scalebox{0.5}{
	\begin{tikzpicture}
		\node(N1) at (0,0){$\bullet$};
		\node(N2) at (0.5,-1.5){$\bullet$};
		\node(N3) at (1,-1){$\bullet$};
		\node(N4) at (1.5,-0.5){$\bullet$};
		
		\draw (N4.center) -- (N1.center);
		\draw (N3.center) -- (N4.center);
		\draw (N2.center) -- (N3.center);
	\end{tikzpicture}
	}
};

\node(T11) at (8,-5){
	\scalebox{0.5}{
	\begin{tikzpicture}
		\node(N1) at (0,0){$\bullet$};
		\node(N2) at (0.5,-1){$\bullet$};
		\node(N3) at (1,-1.5){$\bullet$};
		\node(N4) at (1.5,-0.5){$\bullet$};
		
		\draw (N4.center) -- (N1.center);
		\draw (N2.center) -- (N4.center);
		\draw (N3.center) -- (N2.center);
	\end{tikzpicture}
	}
};

\node(T12) at (2,-5){
	\scalebox{0.5}{
	\begin{tikzpicture}
		\node(N1) at (0,0){$\bullet$};
		\node(N2) at (0.5,-1){$\bullet$};
		\node(N3) at (1,-0.5){$\bullet$};
		\node(N4) at (1.5,-1){$\bullet$};
		
		\draw (N3.center) -- (N1.center);
		\draw (N4.center) -- (N3.center);
		\draw (N2.center) -- (N3.center);
	\end{tikzpicture}
	}
};

\node(T13) at (6,-5.5){
	\scalebox{0.5}{
	\begin{tikzpicture}
		\node(N1) at (0,0){$\bullet$};
		\node(N2) at (0.5,-0.5){$\bullet$};
		\node(N3) at (1,-1.5){$\bullet$};
		\node(N4) at (1.5,-1){$\bullet$};
		
		\draw (N2.center) -- (N1.center);
		\draw (N4.center) -- (N2.center);
		\draw (N3.center) -- (N4.center);
	\end{tikzpicture}
	}
};

\node(T14) at (4,-6){
	\scalebox{0.5}{
	\begin{tikzpicture}
		\node(N1) at (0,0){$\bullet$};
		\node(N2) at (0.5,-0.5){$\bullet$};
		\node(N3) at (1,-1){$\bullet$};
		\node(N4) at (1.5,-1.5){$\bullet$};
		
		\draw (N2.center) -- (N1.center);
		\draw (N3.center) -- (N2.center);
		\draw (N4.center) -- (N3.center);
	\end{tikzpicture}
	}
};

\draw (T2) -- (T1);
\draw (T3) -- (T1);
\draw (T4) -- (T2);
\draw (T5) -- (T3);
\draw (T5) -- (T4);
\draw (T6) -- (T1);
\draw (T7) -- (T2);
\draw (T7) -- (T6);
\draw (T8) -- (T3);
\draw (T9) -- (T6);
\draw (T9) -- (T8);
\draw (T10) -- (T4);
\draw (T11) -- (T5);
\draw (T11) -- (T10);
\draw (T12) -- (T7);
\draw (T12) -- (T10);
\draw (T13) -- (T8);
\draw (T13) -- (T11);
\draw (T14) -- (T9);
\draw (T14) -- (T12);
\draw (T14) -- (T13);
\end{tikzpicture}
\end{tabular}
}
\caption{Tamari lattices of size 3 and 4.}

\label{fig:tamari-order}

\end{figure}
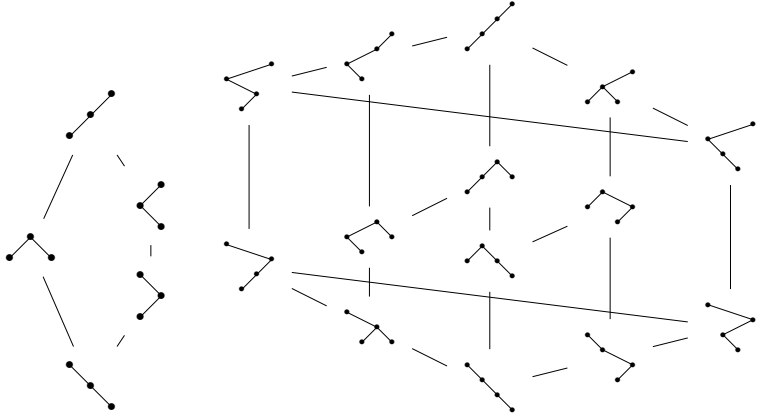

The Tamari lattice is a quotient of the weak order on permutations
\cite{plactic-monoid}. To understand the relation between the two
orders, we need the notion of \emph{binary search tree}.

\begin{Definition}
\label{def:binary-search-tree}
A \emph{binary search tree} is a labelled binary tree where for each node of
label $k$, any label in his left (resp. right) subtree is lower
than or equal to (resp. greater than) $k$.
\end{Definition}

Figure \ref{fig:bst-example} shows an example of binary search
tree. For a given binary tree $T$ of size $n$, there is a unique
labelling of $T$ with $1, \dots, n$ such that $T$ is a binary search
tree. Such a labelled tree can then be seen as a poset. For example,
the tree
\begin{center}
\scalebox{0.7}{
  \begin{tikzpicture}
    \node(N1) at (0,-0.5){1};
    \node(N2) at (0.5,0){2};
    \node(N3) at (1,-0.5){3};
    
    \draw (N1) -- (N2);
    \draw (N3) -- (N2);
  \end{tikzpicture}
  }
\end{center}
is the poset where $1$ and $3$ are smaller than $2$. We write $1 \prec
2$ and $3 \prec 2$. A linear extension of this poset is a permutation
where if $a \prec b$ in the poset, then the number $a$ is before $b$
in the permutation. For the above tree, the linear extensions are
$132$ and $312$. The sets of permutations corresponding to the linear
extensions of the binary trees of size $n$ form a partition of
$\mathfrak{S}_n$ and more precisely, each set is an interval of the
right weak order on permutations called a \emph{sylvester class} and
the Tamari order is a lattice on these classes
\cite{plactic-monoid}. See Figures \ref{fig:bst-example} for examples
of sylvester classes.

\begin{figure}[ht]
\centering
\begin{tabular}{cl | cc}
\begin{tikzpicture}
\node(T1) at (0,-2){1};
\node(T2) at (0.5,-1){2};
\node(T3) at (1,-2){3};
\node(T4) at (1.5,0){4};
\node(T5) at (2,-1){5};
\draw(T1)--(T2);
\draw(T3)--(T2);
\draw(T2)--(T4);
\draw(T5)--(T4);
\end{tikzpicture}
&
\scalebox{0.7}{
\begin{tikzpicture}
\node(P13254) at (1,0){13254};
\node(P31254) at (0,-1){31254};
\node(P13524) at (2,-1){13524};
\node(P31524) at (0,-2){31524};
\node(P15324) at (2,-2){15324};
\node(P35124) at (0,-3){35124};
\node(P51324) at (2,-3){51324};
\node(P53124) at (1,-4){53124};
\draw(P53124) -- (P51324);
\draw(P53124) -- (P35124);
\draw(P51324) -- (P15324);
\draw(P35124) -- (P31524);
\draw(P15324) -- (P13524);
\draw(P31524) -- (P31254);
\draw(P31524) -- (P13524);
\draw(P13524) -- (P13254);
\draw(P31254) -- (P13254);
\end{tikzpicture}
}
&
\begin{tikzpicture}
\node(P1) at (0,1.5){123};
\node(P2) at (0.5,0.5){213};
\node(P3) at (-0.5,0.5){132};
\node(P4) at (0.5,-0.5){231};
\node(P5) at (-0.5,-0.5){312};
\node(P6) at (0,-1.5){321};

\draw (P1) -- (P2);
\draw (P1) -- (P3);
\draw (P2) -- (P4);
\draw (P3) -- (P5);
\draw (P4) -- (P6);
\draw (P5) -- (P6);

\draw (0,1.5) ellipse (0.4cm and 0.4cm);
\draw (-0.5,0) ellipse (0.6cm and 1cm) ;
\draw (0.5,0.5) ellipse (0.4cm and 0.4cm) ;
\draw (0.5,-0.5) ellipse (0.4cm and 0.4cm) ;
\draw (0,-1.5) ellipse (0.4cm and 0.4cm) ;
\end{tikzpicture}
&
\scalebox{0.8}{
\begin{tikzpicture}
\node(T1) at (1,0) {
  \scalebox{0.7}{
    \begin{tikzpicture}
      \node(N1) at (0,-1){1};
      \node(N2) at (0.5,-0.5){2};
      \node(N3) at (1,0){3};
      
      \draw (N1) -- (N2);
      \draw (N2) -- (N3);
    \end{tikzpicture}
  }
};
\node(T2) at (0,-2.2) {
  \scalebox{0.7}{
    \begin{tikzpicture}
      \node(N1) at (0,-0.5){1};
      \node(N2) at (0.5,0){2};
      \node(N3) at (1,-0.5){3};
      
      \draw (N1) -- (N2);
      \draw (N3) -- (N2);
    \end{tikzpicture}
  }
};

\node(T3) at (2,-1.5) {
  \scalebox{0.7}{
    \begin{tikzpicture}
      \node(N1) at (0,-0.5){1};
      \node(N2) at (0.5,-1){2};
      \node(N3) at (0.5,0){3};
      
      \draw (N1) -- (N3);
      \draw (N2) -- (N1);
    \end{tikzpicture}
  }
};

\node(T4) at (2,-3) {
  \scalebox{0.7}{
    \begin{tikzpicture}
      \node(N1) at (0,0){1};
      \node(N2) at (0,-1){2};
      \node(N3) at (0.5,-0.5){3};
      
      \draw (N3) -- (N1);
      \draw (N2) -- (N3);
    \end{tikzpicture}
  }
};

\node(T5) at (1,-4.5) {
  \scalebox{0.7}{
    \begin{tikzpicture}
      \node(N1) at (0,0){1};
      \node(N2) at (0.5,-0.5){2};
      \node(N3) at (1,-1){3};
      
      \draw (N3) -- (N2);
      \draw (N2) -- (N1);
    \end{tikzpicture}
  }
};

\draw (T1) -- (T2);
\draw (T1) -- (T3);
\draw (T3) -- (T4);
\draw (T4) -- (T5);
\draw (T2) -- (T5);

\end{tikzpicture}
}
\end{tabular}

\caption{On the left: a binary search tree and its corresponding sylvester class, and on the right: the sylvester classes of the permutohedron of size 3, with the corresponding binary search trees.}
\label{fig:bst-example}
\end{figure}
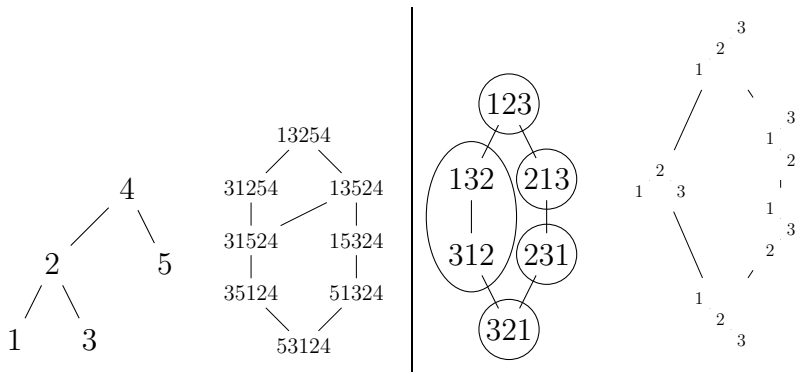

\subsection{Construction of interval-posets}
\label{sec:interval-posets}

We now introduce a more general object, called the \emph{interval-poset},
which is in bijection with intervals of the Tamari order. Let us first define two bijections between binary search trees and forests of planar trees. A binary search tree $T$ is a poset  containing two kinds of relations: when $a$ is in the left subtree of $b$, we have an increasing relation $a<b$ and $a \prec_T b$ and when $b$ is in the right subtree of $a$, we have a decreasing relation $b>a$ and $b \prec_T a$. The two bijections we define consist in keeping only increasing (resp. decreasing) relations of the poset.

\begin{Definition}
\label{def:LSRB-RSLB}

The \emph{increasing forest}\footnote{Note that what we call increasing means increasing labels from the leaf to the root and not from the root to the leaf as it is often the case.} (noted $\inc$) of a binary search tree $T$ is a forest poset on the nodes of $T$ containing  only increasing relations and such that:

\begin{equation}
a \prec_{\inc(T)} b \Leftrightarrow a < b ~\text{ and }~ a \prec_{T} b.
\end{equation}

It is equivalent to the following construction:

\begin{itemize}
  \item if a node labelled $x$ has a left son labelled $y$ in $T$
    then the node $x$ have a son $y$ in $F$;
  \item if a node labelled $x$ has a right son labelled $y$ in $T$
    then the node $x$ have a brother $y$ in $T'$.
\end{itemize}

In the same way, one can define the \emph{decreasing forest} (noted
$\dec$) by switching the roles of the right and left son in the previous
construction or, in terms of posets:

\begin{equation}
b \prec_{\dec(F)} a \Leftrightarrow b < a ~\text{ and }~ b \prec_{T} a.
\end{equation}

\end{Definition}

\begin{figure}[ht]

\begin{center}

\begin{tabular}{|c|c|c|}
\hline
tree $T$ & $\inc(T)$ & $\dec(T)$ \\
\hline
\scalebox{0.7}{
\begin{tikzpicture}
  \node(N1) at (-2,-1){1};
  \node(N2) at (-1.5,-3){2};
  \node(N3) at (-1,-2){3};
  \node(N4) at (-0.5,-3){4};
  \node(N5) at (0,0){5};
  \node(N6) at (1,-2){6};
  \node(N7) at (2,-1){7};
  \node(N8) at (2.5,-3){8};
  \node(N9) at (3,-2){9};
  \node(N10) at (3.5,-3){10};

  \draw (N1) -- (N5);
  \draw (N7) -- (N5);
  \draw (N3) -- (N1);
  \draw (N6) -- (N7);
  \draw (N9) -- (N7);
  \draw (N2) -- (N3);
  \draw (N4) -- (N3);
  \draw (N8) -- (N9);
  \draw (N10) -- (N9);
\end{tikzpicture}
}
&
\scalebox{0.7}{
\begin{tikzpicture}
  \node(N1) at (-2,-1){1};
  \node(N2) at (-1.5,-2){2};
  \node(N3) at (-1.5,-1){3};
  \node(N4) at (-1,-1){4};
  \node(N5) at (-1.5,0){5};
  \node(N6) at (-0.5,-1){6};
  \node(N7) at (-0.5,0){7};
  \node(N8) at (0.5,-1){8};
  \node(N9) at (0.5,0){9};
  \node(N10) at (1.5,0){10};

  \draw (N1) -- (N5);
  \draw (N3) -- (N5);
  \draw (N4) -- (N5);
  \draw (N6) -- (N7);
  \draw (N8) -- (N9);
  \draw (N2) -- (N3);
\end{tikzpicture}
}
&
\scalebox{0.7}{
\begin{tikzpicture}
  \node(N1) at (-1,0){1};
  \node(N2) at (-1.5,-1){2};
  \node(N3) at (-0.5,-1){3};
  \node(N4) at (-0.5,-2){4};
  \node(N5) at (1,0){5};
  \node(N6) at (0.5,-1){6};
  \node(N7) at (1.5,-1){7};
  \node(N8) at (1,-2){8};
  \node(N9) at (2,-2){9};
  \node(N10) at (2,-3){10};

  \draw (N2) -- (N1);
  \draw (N3) -- (N1);
  \draw (N6) -- (N5);
  \draw (N7) -- (N5);
  \draw (N4) -- (N3);
  \draw (N8) -- (N7);
  \draw (N9) -- (N7);
  \draw (N10) -- (N9);
\end{tikzpicture}
}
\\
\hline

\end{tabular}

\end{center}

\caption{A tree with its increasing and decreasing forests.}

\label{fig:lsrb-rslb}

\end{figure}
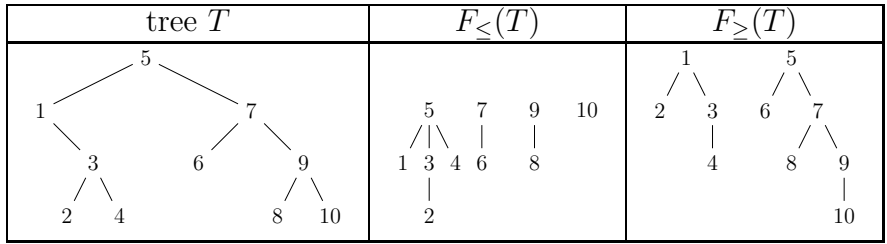

\noindent In Figure~\ref{fig:lsrb-rslb}, we can see a tree $T$
with its decreasing and increasing forests. The linear extensions of
the decreasing and increasing forests are actually initial and
final intervals of the weak order.

\begin{Proposition}

\label{prop:minmax-interval}
The linear extensions of the increasing forest of a tree $T$ is the
union of linear extensions of all trees lower than or equal to $T$
(initial interval) and the linear extensions of the decreasing forest
of $T$ is the union of all trees greater than or equal to $T$ (final
interval).

\end{Proposition}

\begin{proof}[(sketch)]
We just need to recall that $\sigma \leq \mu$ in the weak order means that $coinv(\sigma) \subseteq coinv(\mu)$, where $coinv(\sigma) := \lbrace (\sigma(i), \sigma(j)); i<j, \sigma(i)>\sigma(j) \rbrace$ . It is then easy to see that the linear extension with maximal (resp. minimal) number of co-inversions is the same for $T$ than for $\inc$ (resp. $\dec$). Conversely, if the co-inversions of a permutation $\mu$ are included in the co-inversions of the maximal linear extension of a tree for the weak order, then $\mu$ is a linear extension of $\inc$. The same reasoning can be made for $\dec$.
\end{proof}

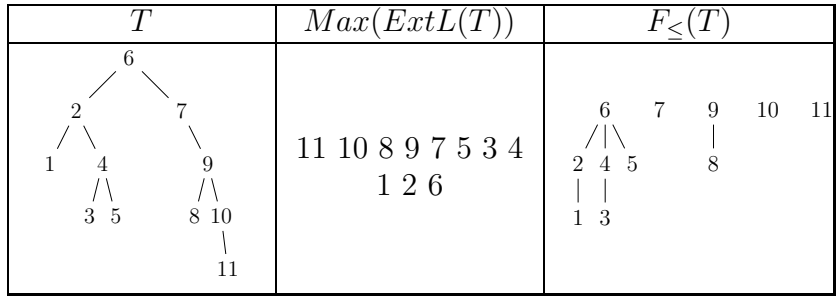
\begin{figure}[ht!]

\begin{center}

\begin{tabular}{|C|C|D|}
\hline
$T$
& 
$Max(ExtL(T))$
&
$\inc(T)$
\\
\hline

\scalebox{0.7}{

\begin{tikzpicture}
  \node(N1) at (-1.5,-2){1};
  \node(N2) at (-1,-1){2};
  \node(N3) at (-0.750,-3){3};
  \node(N4) at (-0.5,-2){4};
  \node(N5) at (-0.250,-3){5};
  \node(N6) at (0,0){6};
  \node(N7) at (1,-1){7};
  \node(N8) at (1.250,-3){8};
  \node(N9) at (1.5,-2){9};
  \node(N10) at (1.750,-3){10};
  \node(N11) at (1.875,-4){11};
  
  \draw (N2) -- (N6);
  \draw (N7) -- (N6);
  \draw (N1) -- (N2);
  \draw (N4) -- (N2);
  \draw (N9) -- (N7);
  \draw (N3) -- (N4);
  \draw (N5) -- (N4);
  \draw (N8) -- (N9);
  \draw (N10) -- (N9);
  \draw (N10) -- (N11);
\end{tikzpicture}
}
& 
11 10 8 9 7 5 3 4 1 2 6
&
\scalebox{0.72}{
\begin{tikzpicture}
  \node(N1) at (-2.5,-2){1};
  \node(N2) at (-2.5,-1){2};
  \node(N3) at (-2,-2){3};
  \node(N4) at (-2,-1){4};
  \node(N5) at (-1.5,-1){5};
  \node(N6) at (-2,0){6};
  \node(N7) at (-1,0){7};
  \node(N8) at (0,-1){8};
  \node(N9) at (0,0){9};
  \node(N10) at (1,0){10};
  \node(N11) at (2,0){11};
  
  \draw (N2) -- (N6);
  \draw (N4) -- (N6);
  \draw (N5) -- (N6);
  \draw (N8) -- (N9);
  \draw (N1) -- (N2);
  \draw (N3) -- (N4);
\end{tikzpicture}
}
\\
\hline
\end{tabular}

\end{center}

\caption{A tree with the maximum of its sylvester class and its
  increasing forest.}

\label{fig:max-sylv}

\end{figure}

An example of this construction can be found in
Figure~\ref{fig:lsrb-rslb} and another example of an increasing forest
is given in Figure~\ref{fig:max-sylv} with its maximal linear
extension. If two trees $T$ and $T'$ are such that $T \leq T'$, then $\dec(T)$
and $\inc(T')$ share some linear extensions (by
Proposition~\ref{prop:minmax-interval}). More precisely, we have $ExtL(\dec(T)) \cap ExtL(\inc(T')) = [Min(ExtL(T)),\allowbreak Max(ExtL(T'))]$. This set corresponds exactly to the linear extensions of the trees of the interval $[T,T']$ in the Tamari order. It is then natural to
construct a poset that would contain relations of both $\dec(T)$ and $\inc(T')$, see Figure \ref{fig:forest-intersection} for an example. We give a characterization of these posets.

\begin{Definition}
  \label{def:interval-poset-definition}
  An \emph{interval-poset} $P$ is a poset such that the following
  conditions hold:
  \begin{itemize}
    \item $a \prec_{P} c$ implies that for all $a < b < c$, we have
      $b \prec_{P} c$, 
    \item $c \prec_{P} a$ implies that for all $a < b < c$,
      we have $b \prec_{P} a$.
  \end{itemize}
\end{Definition}

\begin{figure}[ht]

\begin{center}
\scalebox{0.8}{
\begin{tabular}{|c|c|c|c|c|}
\hline
$T$ & $T'$ & $\dec(T)$ & $\inc(T')$ & $\dec(T) \cap \inc(T')$ \\
\hline
\begin{tikzpicture}
  \node(N1) at (-2,-2){1};
  \node(N2) at (-1,-3){2};
  \node(N3) at (-1,-1){3};
  \node(N4) at (0,0){4};
  
  \draw (N2) -- (N1);
  \draw (N1) -- (N3);
  \draw (N3) -- (N4);
\end{tikzpicture}
&
\begin{tikzpicture}
  \node(N1) at (0,0){1};
  \node(N2) at (0,-2){2};
  \node(N3) at (1,-1){3};
  \node(N4) at (2,-2){4};
  
  \draw (N3) -- (N1);
  \draw (N2) -- (N3);
  \draw (N4) -- (N3);
\end{tikzpicture}
&
\begin{tikzpicture}
  \node(N1) at (-1,0){1};
  \node(N2) at (-1,-1){2};
  \node(N3) at (0,0){3};
  \node(N4) at (1,0){4};
  
  \draw (N2) -- (N1);
\end{tikzpicture}
&
\begin{tikzpicture}
  \node(N1) at (-1,0){1};
  \node(N2) at (0,-1){2};
  \node(N3) at (0,0){3};
  \node(N4) at (1,0){4};
  
  \draw (N2) -- (N3);
\end{tikzpicture}
&
\begin{tikzpicture}
  \node(N1) at (-1,0){1};
  \node(N2) at (-0.5,-1){2};
  \node(N3) at (0,0){3};
  \node(N4) at (1,0){4};

  \draw (N2) -- (N1);
  \draw (N2) -- (N3);
\end{tikzpicture}
\\
\hline 
\end{tabular}
}
\end{center}

\caption{Two trees $T$ and $T'$, their decreasing and increasing
  forest and the interval-poset $[T, T']$. The linear extensions of the interval-poset correspond to the interval $[2134, 4231]$ of the weak order and $2134$ (resp. $4231$) is the minimal (resp. maximal) linear extension of $T$ (resp. $T'$).}

\label{fig:forest-intersection}

\end{figure}

\begin{Proposition}
  \label{prop:tamari-interval-characterization}
  The interval-posets are exactly the posets whose linear extensions
  correspond to Tamari intervals for the weak order.
\end{Proposition}

Indeed, it is easy to see that from an interval-poset, one can build
$F_\leq$ (resp. $F_\geq$) by only considering the increasing relations
(resp. decreasing relations). Conditions of Definition
\ref{def:interval-poset-definition} are necessary and sufficient to
obtain well-defined increasing (resp. decreasing) forests that 
correspond to proper binary search trees.

\subsection{Combinatorial properties of interval-posets}
\label{sec:comb-prop}

Many operations on intervals can be easily done on interval-posets, all with trivial proofs. 

\begin{Proposition}
\label{prop:comb-prop-trivial}
\begin{enumerate}[label=(\roman{*}), ref=(\roman{*})]
\item The intersection between two intervals $I_1$ and $I_2$ is given by the interval-poset $I_3$ containing all relations of $I_1$ and $I_2$. If $I_3$ is a valid poset, then it is a valid interval-poset, otherwise the intersection is empty. 
\label{prop:comb-prop-intersect}
\item An interval $I_1 := \left[ T_1, T_1' \right] $ is contained into an interval $I_2 := \left[ T_2, T_2' \right]$, \emph{i.e.}, $T_1 \geq T_2$ and $T_1' \leq T_2'$, if and only if all relations of the interval-poset $I_1$ are satisfied by the interval-poset $I_2$. 
\label{prop:comb-prop-inclusion}
\item If $I_1 := \left[ T_1, T_1' \right]$ is an interval, then $I_2 = \left[ T_2, T_1' \right]$, $T_2 \geq T_1$, if and only if all relations of the interval-poset $I_1$ are satisfied by $I_2$ and all new relations of $I_2$ are decreasing. Symmetrically, $I_3 = \left[ T_1, T_3 \right]$, $T_3 \leq T_1'$, if and only if all relations of the interval-poset $I_1$ are satisfied by $I_3$ and all new relations of $I_3$ are increasing. 
\label{prop:comb-prop-minmax-inclusion}
\end{enumerate}
\end{Proposition}

\section{Tamari polynomials}
\label{sec:tamari-polynomials}

\subsection{Bilinear form and enumeration}
\label{sec:bilinear-form}

Let $\phi(y)$ be the generating function of Tamari intervals,

\begin{equation}
\label{eq:generating-function}
\phi(y) = 1 + y + 3 y^2 + 13 y^3 + 68 y^4 + \dots~.
\end{equation}
where $y$ counts the number of nodes in the trees or equivalently the number of vertices in the interval-posets. In \cite{Chap}, Chapoton gives a refined version of $\phi$ with a parameter $x$ that counts the number of nodes on the left border of the smaller tree of the interval,

\begin{equation}
\label{eq:refined-generating-function}
\Phi(x,y) = 1 + xy + (x + 2x^2)y^2 + (3x + 5x^2 + 5x^3)y^3 + \dots~.
\end{equation}

We know that an interval-poset $I$ of $[T,T']$ is formed by two forest posets of respectively decreasing relations of $T$ and increasing relations of $T'$. The number of nodes in the left border of $T$ can then be seen as the number of trees in $\dec(T)$, \emph{i.e.}, the poset formed by the decreasing relations of $I$. This way, one can interpret the refined generating function \eqref{eq:refined-generating-function} directly on interval-posets. In \cite[formula (6)]{Chap}, Chapoton gives a functional equation on $\Phi$:\footnote{Our equation is slightly different from the one of \cite[formula (6)]{Chap}. Indeed, the definition of the degree of $x$ differs by one and in our case $\Phi$ also counts the interval of size $0$.}

\begin{equation}
\label{eq:functional-equation}
\Phi(x,y) = xy\Phi(x,y)\frac{x\Phi(x,y) - \Phi(1,y)}{x-1} +1.
\end{equation}

The generating function $\Phi$ is then the solution of 
\begin{equation}
\label{eq:bilinear-equation}
\Phi = B(\Phi,\Phi) +1
\end{equation}
where $B$ is the bilinear form 
\begin{equation}
\label{eq:bilinear}
B(f,g) = xyf(x,y)\frac{xg(x,y) - g(1,y)}{x - 1}.
\end{equation}
By developing \eqref{eq:bilinear-equation}, one then obtains that 
\begin{align}
\Phi &= 1 + B(1,1) + B(B(1,1),1) + B(1,B(1,1)) + \dots \\
&= \sum_T \mathcal{B}_T,
\end{align}
summing over all binary trees $T$, with $\mathcal{B}_T$ being recursively defined by $\mathcal{B}_\emptyset := 1$ and $\mathcal{B}_T := B(\mathcal{B}_L, \mathcal{B}_R)$ where $L$ and $R$ are respectively the left and right children of $T$. By a combinatorial interpretation of $B$, we actually prove that $\mathcal{B}_T$ counts the number of trees smaller than $T$ in the Tamari order. We also obtain a new way of generating intervals and thus prove in a new way that the generating function of the interval satisfies the functional equation \eqref{eq:functional-equation}. Let us define a new operation on interval-posets:

\begin{Definition}
\label{def:composition}
Let $I_1$ and $I_2$ be two interval-posets of size respectively $k_1$ and $k_2$. Then $ \mathbb{B}(I_1, I_2)$ is the formal sum of all interval-posets of size $k_1 + k_2 + 1$ where, 
\begin{enumerate}[label=(\roman{*}), ref=(\roman{*})]
\item the relations between vertices $1, \dots, k_1$ are exactly the ones from $I_1$,
\label{def:composition:cond:I1}
\item the relations between $k_1 + 2, \dots, k_1 + k_2 + 1$ are
  exactly the ones from $I_2$ shifted by $k_1 + 1$,
\label{def:composition:cond:I2}
\item we have $i \prec k_1 +1$ for all $i \leq k_1$,
\label{def:composition:cond:incr}
\item there is no relation $k_1+1 \prec j$ for all $j>k_1+1$.
\label{def:composition:cond:decr}
\end{enumerate}
We call this operation the composition of intervals and extend it by bilinearity to all linear sums of intervals.
\end{Definition}

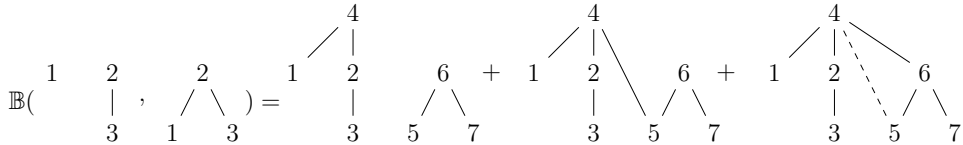
\begin{figure}[ht]
\centering
\scalebox{0.8}{
\begin{tikzpicture}
\node(B) at (0,-0.5){$ \mathbb{B}($};

\begin{scope}[xshift=0.5cm]
\node(T11) at (0,0){1};
\node(T12) at (1,0){2};
\node(T13) at (1,-1){3};
\draw(T13) -- (T12);
\end{scope}

\node(times) at (2,-0.5){$,$};

\begin{scope}[xshift=2.5cm]
\node(T22) at (0.5,0){2};
\node(T21) at (0,-1){1};
\node(T23) at (1,-1) {3};
\draw(T21) -- (T22);
\draw(T23) -- (T22);
\end{scope}

\node(equal) at (4,-0.5){$) =$};

\begin{scope}[xshift=4.5cm]
\node(T1) at (0,0){1};
\node(T2) at (1,0){2};
\node(T3) at (1,-1){3};
\node(T4) at (1,1){4};
\node(T5) at (2,-1){5};
\node(T6) at (2.5,0){6};
\node(T7) at (3,-1){7};
\draw (T1) -- (T4);
\draw (T2) -- (T4);
\draw (T3) -- (T2);
\draw (T5) -- (T6);
\draw (T7) -- (T6);
\end{scope}

\node(plus1) at (7.8,0){$+$};

\begin{scope}[xshift=8.5cm]
\node(T1) at (0,0){1};
\node(T2) at (1,0){2};
\node(T3) at (1,-1){3};
\node(T4) at (1,1){4};
\node(T5) at (2,-1){5};
\node(T6) at (2.5,0){6};
\node(T7) at (3,-1){7};
\draw (T1) -- (T4);
\draw (T2) -- (T4);
\draw (T3) -- (T2);
\draw (T5) -- (T6);
\draw (T5) -- (T4);
\draw (T7) -- (T6);
\end{scope}

\node(plus2) at (11.7,0){$+$};

\begin{scope}[xshift=12.5cm]
\node(T1) at (0,0){1};
\node(T2) at (1,0){2};
\node(T3) at (1,-1){3};
\node(T4) at (1,1){4};
\node(T5) at (2,-1){5};
\node(T6) at (2.5,0){6};
\node(T7) at (3,-1){7};
\draw (T1) -- (T4);
\draw (T2) -- (T4);
\draw (T3) -- (T2);
\draw (T5) -- (T6);
\draw[dashed] (T5) -- (T4);
\draw (T6) -- (T4);
\draw (T7) -- (T6);
\end{scope}
\end{tikzpicture}
}
\caption{Composition of interval-posets: the three terms of the sum are obtained by adding respectively no, 1, and 2 decreasing relations between the second poset and the vertex $4$. For the last term, two decreasing relations have been added: $5 \prec 4$ and $6 \prec 4$, the $5\prec4$ relation has been dashed as it is implicit through transitivity.}
\label{fig:composition}
\end{figure}

The sum we obtain by composing interval-posets actually corresponds to all possible ways of adding decreasing relations between the second poset and the new vertex $k_1 + 1$, as seen on Figure \ref{fig:composition}. Especially, there is no relations between vertices $1, \dots, k_1$ and $k_1+2, \dots, k_1+k_2+1$. Indeed, condition \ref{def:composition:cond:incr} makes it impossible to have any relation $j \prec i$ with $i<k_1+1<j$ as this would imply by Definition \ref{def:interval-poset-definition} that $k_1+1 \prec i$. And condition \ref{def:composition:cond:decr} makes it impossible to have $i \prec j$ as this would imply $k_1 + 1 \prec j$.

\begin{Proposition}
\label{prop:combinatorial-equivalence-composition}
Let $I_1$ and $I_2$ be two interval-posets. Let $\mathcal{P}$ be the linear function that associates with an interval-poset its monomial $x^{trees}y^{size}$ where the power of $y$ is the number of vertices and the power of $x$ the number of trees obtained by keeping only decreasing relations. Then 

\begin{equation}
\mathcal{P}(\mathbb{B}( I_1,  I_2)) = B(\mathcal{P}(I_1), \mathcal{P}(I_2)).
\end{equation}

\end{Proposition}

As an example, in Figure \ref{fig:composition}, $\mathcal{P}(I_1) = \mathcal{P}(I_2) = x^2 y^3$. And we have $ \mathcal{P}(\mathbb{B}( I_1, I_2)) = x^5y^7 + x^4y^7 + x^3y^7 = B(x^2y^3, x^2y^3)$.

\begin{proof}
If $I_1$ and $I_2$ are two interval-posets of size respectively $k_1$ and $k_2$, we have by definition that all interval-posets of $\mathbb{B}(I_1, I_2)$ are of size $k_1 + k_2 + 1$. Thus the power of $y$ is the same in $B(\mathcal{P}(I_1), \mathcal{P}(I_2))$ and in $\mathcal{P}( \mathbb{B}( I_1, I_2))$ and we only have to consider the polynomial in $x$.

Let us assume that $I_1$ and $I_2$ contain respectively $n$ and $m$ trees formed by decreasing relations. The $n$ trees of $I_1$ are kept unchanged on all terms of the result as no decreasing relation is added to the vertices $1, \dots, k_1$. Now, we call $v_1 < \dots < v_m$ the root vertices of the trees of $I_2$ shifted by $k_1 + 1$. By construction, $k_1+1 < v_1$, and this new vertex can either become a new root or a root to some of the previous trees. If we have $v_j \prec k_1 + 1$, by definition of an interval-poset, we also have $v_i \prec k_1 + 1$ for all $i<j$. The $m$ trees of $I_2$ can then be replaced by either $m+1,m,\dots,2,$ or $1$ trees, which mean the monomial $x^m$ of $\mathcal{P}(I_2)$ becomes $x+x^2 + \dots + x^{m+1}$ in the composition. So,

\begin{align}
\mathcal{P}( \mathbb{B}( I_1, I_2)) &= y(x^n y^{k_1})y^{k_2}x \frac{x^{m+1} - 1}{x- 1} \\
&= B( \mathcal{P}(I_1), \mathcal{P}(I_2)).
\end{align}

\end{proof}

To prove now that the generating function of the intervals is the solution of the bilinear equation \eqref{eq:bilinear-equation}, we only need the following proposition.

\begin{Proposition}
Let $I$ be an interval-poset, then, there is exactly one pair of intervals $I_1$ and $I_2$ such that $I$ appears in the composition $ \mathbb{B}( I_1, I_2)$. 
\end{Proposition}

\begin{proof}
Let $I$ be an interval-poset of size $n$ and let $k$ be the vertex of $I$ with maximal label such that $i \prec k$ for all $i < k$. The vertex $1$ satisfies this property, so one can always find such a vertex. We prove that $I$ only appears in the composition of $I_1$ by $I_2$, where $I_1$ is formed by the vertices and relations of $1, \dots, k -1$ and $I_2$ is formed by the re-normalized vertices and relations of $k + 1, \dots, n$. Note that one or both of these intervals can be of size 0.

Conditions \ref{def:composition:cond:I1}, \ref{def:composition:cond:I2}, and \ref{def:composition:cond:incr} of Definition \ref{def:composition} are clearly satisfied by construction. If condition \ref{def:composition:cond:decr} is not satisfied, it means that we have a relation $k \prec j$ with $j>k$. Then, by definition of an interval-poset, we also have $\ell \prec j$ for all $k<l<j$ and by definition of $k$, we have $i\prec k \prec j$ for all $i<k$, so for all $i<j$, we have $i \prec j$. This is not possible as $k$ has been chosen to be maximal among vertices with this property. 

This proves that $I$ appears in the composition of $I_1$ by $I_2$. Now, if $I$ appears in $\mathbb{B}(I_1',I_2')$, the vertex $k' = \vert I_1 ' \vert + 1$ is by definition the vertex where for all $i<k'$, we have $i \prec k'$ and for all $j>k'$, we have $k' \nprec j$, this is exactly the definition of $k$. So $k' = k$ which makes $I_1' = I_1$ and $I_2' = I_2$.
\end{proof}

\subsection{Main result}
\label{sec:main-result}

This composition operation on intervals is an analogue of the usual composition of binary trees that adds a root node to two given binary trees. In our case, a tree $T$ is replaced by a sum of intervals $[T',T]$.

\begin{Proposition}
\label{prop:sum-composition}
Let $T := k(T_1,T_2)$ be a binary tree and $S := \sum_{T' \leq T} [T',T]$. Then, if $S_1 := \sum_{T_1' \leq T_1} [T_1',T_1]$ and $S_2 := \sum_{T_2' \leq T_2} [T_2',T_2]$, we have $S = \mathbb{B}(S_1,S_2)$.
\end{Proposition}

With this new proposition, Theorem \ref{thm:smaller-trees} would be fully proven by induction on the size of the tree. The initial case is trivial, and then if we assume that $\mathcal{P}(S_1) = \mathcal{B}_{T_1}(x)$ and $\mathcal{P}(S_2) = \mathcal{B}_{T_2}(x)$, Proposition \ref{prop:combinatorial-equivalence-composition} tells us that $\mathcal{P}(\mathbb{B}(S_1,S_2)) = B(\mathcal{B}_{T_1}, \mathcal{B}_{T_2})$.

\begin{proof}
Let $T$ be a binary tree of size $n$. The initial interval $\mathcal{T} = [T_0, T]$, is given by the increasing bijection of Definition \ref{def:LSRB-RSLB}, it is a poset containing only increasing relations. By Proposition \ref{prop:comb-prop-trivial}, \ref{prop:comb-prop-minmax-inclusion}, the sum of all intervals $[T',T]$ is given by all possible ways of adding decreasing edges to the poset $\mathcal{T}$.

The increasing poset $\mathcal{T}$ can be formed recursively from the increasing posets $\mathcal{T}_1$ and $\mathcal{T}_2$ of the subtrees $T_1$ and $T_2$ as shown in Figure \ref{fig:recursive-forest}.  The new vertex $k = \vert T_1 \vert +1$ is placed so that $i \prec k$ for all $i \in \mathcal{T}_1$ and the vertices of $\mathcal{T}_2$ are just shifted by $k$. Now, let $I$ be an interval of the sum $S$, $I$ contains the poset $\mathcal{T}$ and some extra decreasing relations. Let $I_1$ and $I_2$ be the subposets formed respectively by vertices $1,\dots,k-1$, and $k+1, \dots, n$. By construction, the posets $I_1$ and $I_2$ contain respectively the forest posets $\mathcal{T}_1$ and $\mathcal{T}_2$ and some extra decreasing relations. This means that $I_1$ appears in $S_1$ and $I_2$ appears in $S_2$. And we have that $I$ appears in $\mathbb{B}(I_1,I_2)$. Indeed, conditions \ref{def:composition:cond:I1} and \ref{def:composition:cond:I2} of Definition \ref{def:composition} are true by construction and conditions \ref{def:composition:cond:incr}   and \ref{def:composition:cond:decr} are true because the increasing relations of $I$ are exactly the ones of $\mathcal{T}$. 

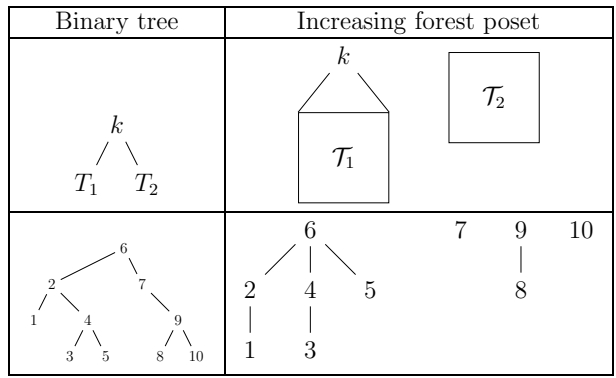
\begin{figure}[ht]
\centering
\scalebox{0.8}{
\begin{tabular}{|c|c|}
\hline
Binary tree
&
Increasing forest poset
\\
\hline
\begin{tikzpicture}
\node(k) at (0.5,0){$k$};
\node(T1) at (0,-1){$T_1$};
\node(T2) at (1,-1){$T_2$};
\draw (T1) -- (k);
\draw (T2) -- (k);
\end{tikzpicture} &
\begin{tikzpicture}
\node(k) at (0,0){$k$};
\node[draw, minimum size = 1.5cm](T1) at (0,-1.7){$\mathcal{T}_1$};
\node[draw, minimum size = 1.5cm](T2) at (2.5,-0.7){$\mathcal{T}_2$};
\draw (k) -- (T1.north east);
\draw (k) -- (T1.north west);
\end{tikzpicture}
\\
\hline
\scalebox{0.6}{
\begin{tikzpicture}
\node(T6) at (2.5,0){6};
\node(T2) at (0.5,-1){2};
\node(T1) at (0,-2){1};
\node(T4) at (1.5,-2){4};
\node(T3) at (1,-3){3};
\node(T5) at (2,-3){5};
\node(T7) at (3,-1){7};
\node(T9) at (4,-2){9};
\node(T8) at (3.5,-3){8};
\node(T10) at (4.5,-3){10};
\draw (T1) -- (T2);
\draw (T3) -- (T4);
\draw (T5) -- (T4);
\draw (T4) -- (T2);
\draw (T2) -- (T6);
\draw (T8) -- (T9);
\draw (T10) -- (T9);
\draw (T9) -- (T7);
\draw (T7) -- (T6);
\end{tikzpicture}
}
&
\begin{tikzpicture}
\node(T6) at (1,0){6};
\node(T2) at (0,-1){2};
\node(T4) at (1,-1){4};
\node(T5) at (2,-1){5};
\node(T1) at (0,-2){1};
\node(T3) at (1,-2){3};
\node(T7) at (3.5,0){7};
\node(T9) at (4.5,0){9};
\node(T10) at (5.5,0){10};
\node(T8) at (4.5,-1){8};
\draw (T2) -- (T6);
\draw (T4) -- (T6);
\draw (T5) -- (T6);
\draw (T1) -- (T2);
\draw (T3) -- (T4);
\draw (T8) -- (T9);
\end{tikzpicture}
\\
\hline
\end{tabular}
}
\caption{The recursive construction of $\mathcal{T}$ from $\mathcal{T}_1$ and $\mathcal{T}_2$.}
\label{fig:recursive-forest}
\end{figure}

Conversely, if $I_1$ and $I_2$ are two elements of respectively $S_1$ and $S_2$, their increasing relations are exactly the ones from respectively $\mathcal{T}_1$ and $\mathcal{T}_2$ which makes all interval-posets $I$ of $\mathbb{B}(I_1, I_2)$ an element of $S$. Indeed, by definition of the composition, the increasing relations of $I$ are exactly the ones of $\mathcal{T}$.
\end{proof}

For a given tree $T$ (with increasing poset $\mathcal{T}$), the coefficient of the monomial with maximal degree in $x$ in $\mathcal{B}_T$ is always 1. It corresponds to the minimal tree of the Tamari order, or to the interval with no decreasing relations, \emph{i.e.}, $\mathcal{T}$. The interval with the maximal number of decreasing relations corresponds to $[T,T]$. An example of $\mathcal{B}_T$ and of the computation of smaller trees is presented in Figure \ref{fig:BTExample}.

\begin{figure}[ht]
\scalebox{0.8}{
\begin{tabular}{cc}
\begin{tabular}{c}
\begin{tikzpicture}
\node(T1) at (0,-2){1};
\node(T2) at (0.5,-3){2};
\node(T3) at (1,-1){3};
\node(T4) at (1.5,0){4};
\node(T5) at (2,-2){5};
\node(T6) at (2.5,-1){6};
\draw (T2) -- (T1);
\draw (T1) -- (T3);
\draw (T3) -- (T4);
\draw (T5) -- (T6);
\draw (T6) -- (T4);
\end{tikzpicture}
\\
\shortstack[l]{
$\mathcal{B}_2 = x$
\\
$\mathcal{B}_1 = x + x^2$
\\
$\mathcal{B}_3 = x^2 + x^3$
\\
$\mathcal{B}_5 = x$
\\
$\mathcal{B}_6 = x^2$
\\
$\mathcal{B}_4 = x^3 + 2x^4 + 2x^5 + x^6$
}
\end{tabular}
&
\begin{tabular}{|c|c||c|c|}
\hline
$T' \leq T$
&
$[T',T]$
&
$T' \leq T$
&
$[T',T]$
\\
\hline
\scalebox{0.6}{
\begin{tikzpicture}
\node(T1) at (0,-4){1};
\node(T2) at (0.5,-3){2};
\node(T3) at (1,-2){3};
\node(T4) at (1.5,-1){4};
\node(T5) at (2,0){5};
\node(T6) at (2.5,1){6};
\draw (T1) -- (T2);
\draw (T2) -- (T3);
\draw (T3) -- (T4);
\draw (T4) -- (T5);
\draw (T5) -- (T6);
\end{tikzpicture}
}
&
\begin{tikzpicture}
\node(T4) at (0.5,0){4};
\node(T3) at (0.5,-1){3};
\node(T1) at (0,-2){1};
\node(T2) at (1,-2){2};
\node(T6) at (1.5,0){6};
\node(T5) at (1.5,-1){5};
\draw (T1) -- (T3);
\draw (T2) -- (T3);
\draw (T3) -- (T4);
\draw (T5) -- (T6);
\end{tikzpicture}
&
\scalebox{0.6}{
\begin{tikzpicture}
\node(T1) at (0,-3){1};
\node(T2) at (0.5,-4){2};
\node(T3) at (1,-2){3};
\node(T4) at (1.5,-1){4};
\node(T5) at (2,0){5};
\node(T6) at (2.5,1){6};
\draw (T2) -- (T1);
\draw (T1) -- (T3);
\draw (T3) -- (T4);
\draw (T4) -- (T5);
\draw (T5) -- (T6);
\end{tikzpicture}
}
&
\begin{tikzpicture}
\node(T4) at (0.5,0){4};
\node(T3) at (0.5,-1){3};
\node(T1) at (0,-2){1};
\node(T2) at (1,-3){2};
\node(T6) at (1.5,0){6};
\node(T5) at (1.5,-1){5};
\draw (T1) -- (T3);
\draw[dashed] (T2) -- (T3);
\draw (T2) -- (T1);
\draw (T3) -- (T4);
\draw (T5) -- (T6);
\end{tikzpicture}
\\
\hhline{|=|=#=|=|}
\scalebox{0.6}{
\begin{tikzpicture}
\node(T1) at (0,-3){1};
\node(T2) at (0.5,-2){2};
\node(T3) at (1,-1){3};
\node(T4) at (1.5,0){4};
\node(T5) at (2,-1){5};
\node(T6) at (2.5,1){6};
\draw (T1) -- (T2);
\draw (T2) -- (T3);
\draw (T3) -- (T4);
\draw (T5) -- (T4);
\draw (T4) -- (T6);
\end{tikzpicture}
}
&
\begin{tikzpicture}
\node(T4) at (0.5,0){4};
\node(T3) at (0.5,-1){3};
\node(T1) at (0,-2){1};
\node(T2) at (1,-2){2};
\node(T6) at (1.5,0){6};
\node(T5) at (1.5,-1){5};
\draw (T1) -- (T3);
\draw (T2) -- (T3);
\draw (T3) -- (T4);
\draw (T5) -- (T6);
\draw (T5) -- (T4);
\end{tikzpicture}
&
\scalebox{0.6}{
\begin{tikzpicture}
\node(T1) at (0,-2){1};
\node(T2) at (0.5,-3){2};
\node(T3) at (1,-1){3};
\node(T4) at (1.5,0){4};
\node(T5) at (2,-1){5};
\node(T6) at (2.5,1){6};
\draw (T2) -- (T1);
\draw (T1) -- (T3);
\draw (T3) -- (T4);
\draw (T5) -- (T4);
\draw (T4) -- (T6);
\end{tikzpicture}
}
&
\begin{tikzpicture}
\node(T4) at (0.5,0){4};
\node(T3) at (0.5,-1){3};
\node(T1) at (0,-2){1};
\node(T2) at (1,-3){2};
\node(T6) at (1.5,0){6};
\node(T5) at (1.5,-1){5};
\draw (T1) -- (T3);
\draw[dashed] (T2) -- (T3);
\draw (T2) -- (T1);
\draw (T3) -- (T4);
\draw (T5) -- (T6);
\draw (T5) -- (T4);
\end{tikzpicture}
\\
\hhline{|=|=#=|=|}
\scalebox{0.6}{
\begin{tikzpicture}
\node(T1) at (0,-3){1};
\node(T2) at (0.5,-2){2};
\node(T3) at (1,-1){3};
\node(T4) at (1.5,0){4};
\node(T5) at (2,-2){5};
\node(T6) at (2.5,-1){6};
\draw (T1) -- (T2);
\draw (T2) -- (T3);
\draw (T3) -- (T4);
\draw (T5) -- (T6);
\draw (T6) -- (T4);
\end{tikzpicture}
}
&
\begin{tikzpicture}
\node(T4) at (0.5,0){4};
\node(T3) at (0.5,-1){3};
\node(T1) at (0,-2){1};
\node(T2) at (1,-2){2};
\node(T6) at (1.5,-1){6};
\node(T5) at (1.5,-2){5};
\draw (T1) -- (T3);
\draw (T2) -- (T3);
\draw (T3) -- (T4);
\draw (T5) -- (T6);
\draw[dashed] (T5) -- (T4);
\draw (T6) -- (T4);
\end{tikzpicture}
&
\scalebox{0.6}{
\begin{tikzpicture}
\node(T1) at (0,-2){1};
\node(T2) at (0.5,-3){2};
\node(T3) at (1,-1){3};
\node(T4) at (1.5,0){4};
\node(T5) at (2,-2){5};
\node(T6) at (2.5,-1){6};
\draw (T2) -- (T1);
\draw (T1) -- (T3);
\draw (T3) -- (T4);
\draw (T5) -- (T6);
\draw (T6) -- (T4);
\end{tikzpicture}
}
&
\begin{tikzpicture}
\node(T4) at (0.5,0){4};
\node(T3) at (0.5,-1){3};
\node(T1) at (0,-2){1};
\node(T2) at (1,-3){2};
\node(T6) at (1.5,-1){6};
\node(T5) at (1.5,-2){5};
\draw (T1) -- (T3);
\draw[dashed] (T2) -- (T3);
\draw (T2) -- (T1);
\draw (T3) -- (T4);
\draw (T5) -- (T6);
\draw[dashed] (T5) -- (T4);
\draw (T6) -- (T4);
\end{tikzpicture}
\\
\hline
\end{tabular}
\end{tabular}
}
\caption{Example of the computation of $\mathcal{B}_T$ and list of all smaller trees with associated intervals }
\label{fig:BTExample}
\end{figure}
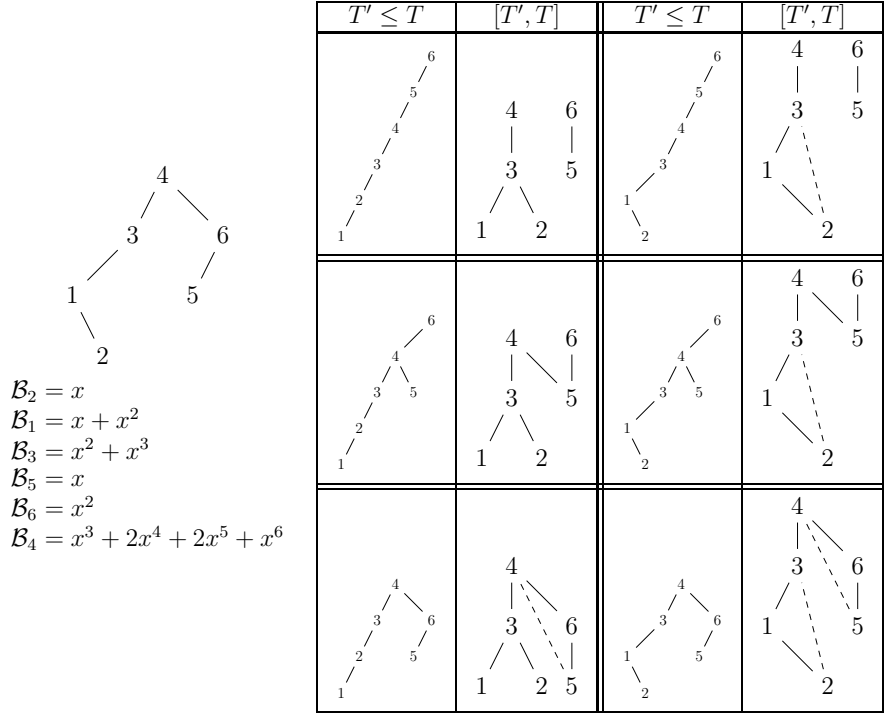

\section{Final comments}
\label{sec:final}

\subsection{Bivariate polynomials}

In some very recent work \cite{ChapBiVar}, Chapoton computed some bivariate polynomials that seem to be similar to the ones we study. By computing the first polynomials of  \cite[formula (7)]{ChapBiVar}, one notices \cite{BiVar} that for $b=1$ and $t = 1 - 1/x$ is equal to $\mathcal{B}_T(x)$, where $T$ is a binary tree with no left subtree. The non planar rooted tree corresponding to $T$ is the non planar version of the tree given by the decreasing bijection of Definition \ref{fig:lsrb-rslb}, \emph{i.e.}, transforming left children of a node into its brothers. 

A $b$ parameter can be also be added to our formula. For an interval $[T',T]$, it is either the number of nodes in $T'$ which have a right subtree, or in the interval-poset the number of nodes $x$ with a relation $y \prec x$ and $y>x$. By a generalization of the linear function $\mathcal{P}$, one can associate a monomial in $b$, $x$, and $y$ with each interval-poset. The bilinear form now reads:

\begin{equation}
B(f,g) = y \left( xbf \frac{x g - g_{x=1}}{x - 1} -bxfg + xfg \right),
\end{equation}
where $f$ and $g$ are polynomials in $x$, $b$, and $y$. Proposition~\ref{prop:combinatorial-equivalence-composition} still holds, since a node with a decreasing relation is added in all terms of the composition but one. As an example, in Figure \ref{fig:composition}, one has $B(y^3x^2b,y^3x^2b) = y^7(x^5b^2 + x^4b^3 + x^3b^3)$.

With this definition of the parameter $b$, the bivariate polynomials $\mathcal{B} 
_T(x,b)$ where $T$ has no left subtree seem to be exactly the ones computed by Chapoton in \cite{ChapBiVar} when taken on $t = 1 - 1/x$. This correspondence and its meaning in terms of algebra and combinatorics should be explored in some future work. 

\subsection{$m$-Tamari}
 
The Tamari lattice on binary trees can also be described in terms of Dyck paths. A Dyck path is a path on the grid formed by north and east steps, starting at $(0,0)$ and ending at $(n,n)$ and never going under the diagonal. One obtains a Dyck path from a binary tree by reading it in postfix order and writing a north step for each empty tree (also called leaf) and an east step for each node, and by ignoring the first leaf. As an example, the binary tree of Figure \ref{fig:BTExample} gives the following path: $N,N,E,E,N,E,N,N,E,N,E,E$. The rotation consists in switching an east step $e$ (immediately followed by a north step) with the shortest translated Dyck path starting right after $e$. One can now consider paths that end in $(mn,n)$ and stay above the line $x = my$, called \emph{$m$-ballot} paths and the same rotation operation will also give a lattice \cite{BergmTamari}.

It is called the $m$-Tamari lattice, a formula counting the number of intervals was conjectured in \cite{BergmTamari} and was proven recently in \cite{mTamari}. The authors use a functional equation that is a direct generalization of \eqref{eq:bilinear-equation}. Let $\Phi_m(x,y)$ be the generating function of intervals of the $m$-Tamari lattice   where $y$ is the size $n$ and $x$ a statistic called number of contacts, then \cite[formula (3)]{mTamari} reads

\begin{equation}
\label{eq:mfunctional-equation}
\Phi_m(x,y) = x + B_m(\Phi,\Phi, \dots, \Phi),
\end{equation}
where $B_m$ is a $m$-linear form defined by 
\begin{align}
B_m(f_1, \dots, f_m) &:= xy f_1 \Delta(f_2 \Delta(\dots \Delta(f_m))\dots), \\
\Delta(g) &:= \frac{g(x,y) - g(1,y)}{x - 1}.
\end{align}

Expanding \eqref{eq:functional-equation}, we obtain a sum of $m$-ary trees. A process is described in \cite{mTamari} to associate a $m$-ballot path with a $m$-ary tree: the tree is read in prefix order, from the right to the left and each leaf (resp. node) is coded by an east (resp. north) step. Note that this process is not consistent with the classical bijection between Dyck path and binary trees: a different definition of the rotation is given which slightly changes the Tamari lattice and could be generalized to $m$-Tamari. However, by computer exploration, one notices that Theorem \ref{thm:smaller-trees} still holds: for a given Dyck path, the polynomials obtained by the postfix and prefix tree interpretations of the path are equal. More generally, given a $m$-ballot path $D$, let $T$ be the $m$-ary tree obtained by a prefix reading. Then the polynomial $\mathcal{B}_{m,T}$ of $T$ where $B_m$ is applied to the nodes and $x$ to the leafs counts the number of $m$-ballot paths lower than $D$ in the $m$-Tamari order. We shall prove this result in future work.

\bibliographystyle{plain}
\label{sec:biblio}
\bibliography{fpsac-article}

\begin{thebibliography}{10}

\bibitem{BergmTamari}
F.~{Bergeron} and L.-F. {Preville-Ratelle}.
\newblock {Higher Trivariate Diagonal Harmonics via generalized Tamari Posets}.
\newblock {\em ArXiv preprint, to appear in J. Combinatorics}, May 2011.
\newblock arXiv:1105.3738.

\bibitem{BW}
A.~Bj{\"o}rner and M.~L. Wachs.
\newblock Permutation statistics and linear extensions of posets.
\newblock {\em J. Combin. Theory Ser. A}, 58(1):85--114, 1991.

\bibitem{mTamari}
M.~Bousquet-M{\'e}lou, E.~Fusy, and L.-F. Pr{\'e}ville-Ratelle.
\newblock The number of intervals in the {$m$}-{T}amari lattices.
\newblock {\em Electron. J. Combin.}, 18(2):Paper 31, 26, 2011.

\bibitem{Chap}
F.~Chapoton.
\newblock Sur le nombre d'intervalles dans les treillis de {T}amari.
\newblock {\em S\'em. Lothar. Combin.}, 55:Art. B55f, 18 pp., 2005/07.

\bibitem{ChapBiVar}
F.~{Chapoton}.
\newblock {Flows on rooted trees and the Menous-Novelli-Thibon idempotents}.
\newblock {\em ArXiv preprint}, 2012.
\newblock arXiv:1203.1780.

\bibitem{BiVar}
F.~{Chapoton}, J.-C. {Novelli}, and J.-Y. {Thibon}.
\newblock Private communication.

\bibitem{plactic-monoid}
F.~Hivert, J.-C. Novelli, and J.-Y. Thibon.
\newblock The algebra of binary search trees.
\newblock {\em Theoret. Comput. Sci.}, 339(1):129--165, 2005.

\bibitem{Tamari}
S.~Huang and D.~Tamari.
\newblock Problems of associativity: {A} simple proof for the lattice property
  of systems ordered by a semi-associative law.
\newblock {\em J. Combinatorial Theory Ser. A}, 13:7--13, 1972.

\bibitem{PBT1}
J.-L. Loday and M.~O. Ronco.
\newblock Hopf algebra of the planar binary trees.
\newblock {\em Adv. Math.}, 139(2):293--309, 1998.

\bibitem{Stanley}
R.P. Stanley.
\newblock {\em Enumerative combinatorics. {V}ol. 2}.
\newblock Cambridge University Press, Cambridge, 1999.

\end{thebibliography}

\end{document}